\documentclass[12 pt]{article}
\usepackage{amsmath, amsthm, amssymb,enumerate,verbatim,authblk}
\usepackage{tikz}
\usetikzlibrary{shapes}
\usetikzlibrary{calc}
\usetikzlibrary{positioning}
\usepackage{fullpage}
\usepackage{hyperref}
\usepackage{tikz}

\usetikzlibrary{calc}
\usetikzlibrary{positioning}
\usetikzlibrary{decorations.markings}

\allowdisplaybreaks

\usepackage[includeall,
            vmargin={1in,1in},
            hmargin={1in,0.375in}
            ]{geometry}

\title{Tournaments, $4$-uniform hypergraphs, and an exact extremal result} 
\author[1]{Karen Gunderson\thanks{karen.gunderson@umanitoba.ca}\thanks{The first author was supported in part by the Heilbronn Institute for Mathematical Research}}
\author[2]{Jason Semeraro\thanks{js13525@bristol.ac.uk}}
\affil[1]{\normalsize{University of Manitoba, Department of Mathematics, Winnipeg, Canada}}
\affil[2]{\normalsize{Heilbronn Institute for Mathematical Research, Department of Mathematics, University of Bristol, U.K.}}

\newtheorem{theorem}{Theorem}
\newtheorem{lemma}[theorem]{Lemma}
\newtheorem{proposition}[theorem]{Proposition}

\newtheorem{corollary}[theorem]{Corollary}
\newtheorem{fact}[theorem]{Fact}

\newtheorem*{thm-main1}{Theorem \ref{thm:main}}

\theoremstyle{definition}

\newcommand{\PSL}{\operatorname{PSL}}
\newcommand{\PGL}{\operatorname{PGL}}
\newcommand{\M}{\operatorname{M}}

\newtheorem{definition}[theorem]{Definition}

\newtheorem{question}[theorem]{Question}

\theoremstyle{remark}

\newcommand{\calH}{\mathcal{H}}
\newcommand{\calM}{\mathcal{M}}

\newcommand{\ex}{\operatorname{ex}}
\newtheoremstyle{case}{}{}{\normalfont}{}{\itshape}{\normalfont:}{ }{}

\numberwithin{equation}{section}

\begin{document}

\maketitle

\begin{abstract}
We consider $4$-uniform hypergraphs with the maximum number of hyperedges subject to the condition that every set of $5$ vertices spans either $0$ or exactly $2$ hyperedges and give a construction, using quadratic residues, for an infinite family of such hypergraphs with the maximum number of hyperedges.  Baber has previously given an asymptotically best-possible result using random tournaments.  We give a connection between Baber's result and our construction via Paley tournaments and investigate a `switching' operation on tournaments that preserves hypergraphs arising from this construction.  
\end{abstract}

\section{Introduction}\label{sec:intro}

For any $r\geq 2$, an $r$-uniform hypergraph $\mathcal{H}$, and integer $n$, the \emph{Tur\'{a}n number} for $\mathcal{H}$, denoted $\operatorname{ex}(\mathcal{H}, n)$, is the maximum number of hyperedges in any $r$-uniform hypergraph on $n$ vertices containing no copy of $\mathcal{H}$.  The \emph{Tur\'{a}n density} of $\mathcal{H}$ is defined to be $\pi(\mathcal{H}) = \lim_{n \to \infty} \operatorname{ex}(\mathcal{H}, n) \binom{n}{r}^{-1}$ and a well-known averaging argument shows that this limit always exists.  While the Tur\'{a}n densities of graphs are well-understood and exact Tur\'{a}n numbers are known for some classes of graphs, few exact results for either Tur\'{a}n numbers or densities are known for the cases $r \geq 3$.  The interested reader is directed to the excellent survey by Keevash \cite{pK11} for more details on known Tur\'{a}n results for hypergraphs.

One particular extremal problem, which has a close connection to the Tur\'{a}n density of a particular $3$-uniform hypergraph was considered by Frankl and F\"{u}redi \cite{FF84}.  Recall that a subset of vertices, $X$, in a hypergraph is said to span an edge $A$ if $A \subseteq X$.  In their paper, Frankl and F\"{u}redi considered those $3$-uniform hypergraphs that have the property that any set of $4$ vertices span either $0$ or exactly $2$ hyperedges.   They proved the exact and strongly structural result that any such hypergraph is in one of two classes: either it is the blow-up of a fixed $3$-graph on $6$ vertices with $10$ edges, or else it is isomorphic to a hypergraph obtained by taking vertices as points around a unit circle and then letting the hyperedges be those triples whose convex hull contains the origin (assuming that no two points lie on a line containing the origin).

This result by Frankl and F\"{u}redi has an application to the Tur\'{a}n density of a small $3$-uniform hypergraph.  Let $K_4^-$ be the $3$-uniform hypergraph on $4$ vertices with $3$ edges.  Using a recursive construction based on the $3$-uniform hypergraph on $6$ vertices mentioned above, Frankl and F\"{u}redi show that $\pi(K_4^-) \geq \frac{2}{7}$.    Using flag algebra techniques, Baber and Talbot \cite{BT11} showed that the Tur\'{a}n density for $K_4^-$ is at most $0.2871$.  It is conjectured in this case that the exact value is $2/7$ (\emph{e.g.} \cite{dM03}).  Many further results on extremal numbers for small cliques in hypergraphs can be found, for example, in \cite{BT12, EFR86, F-RV12, FF87, LZ09, MM62, aR10, aS97, jT07}.

In the same paper, Frankl and F\"{u}redi ask about the following more general question.  For any $r \geq 4$, what is the maximum number of hyperedges in an $r$-uniform hypergraph with the property that any set of $r+1$ vertices span 0 or 2 edges?  They point out that the construction given by points on a circle can be generalized to points on a sphere in $r-1$ dimensions with hyperedges being simplices containing the origin. Frankl and F\"{u}redi note that such an $r$-uniform hypergraph has at most $\binom{n}{r}2^{-r+1}(1+o(1))$ hyperedges for $n$ tending to infinity and that a random choice of vertices on the sphere gives this number of hyperedges. 

In this paper, we consider their question in the case $r=4$ and give a construction for an infinite family of $4$-uniform hypergraphs with the property that every set of $5$ vertices spans either $0$ or $2$ hyperedges with the maximum number of hyperedges among all such hypergraphs on the same number of vertices.  One of the main results of this paper is Theorem \ref{thm:main-const} below, whose proof appears in Section \ref{sec:design}. Before stating this result, we introduce and recall some terminology. Firstly, a 4-uniform hypergraph is defined to be \textit{3-transitive} if its automorphism group acts 3-transitively on the vertex set. (An action of a group $G$ on a set $V$ is \textit{$3$-transitive} if given two 3-tuples $(x,y,z)$ and $(x',y',z')$ of elements of $V$ there exists $g \in G$ such that $(xg,yg,zg)=(x',y',z')$). Next, recall that $\PGL(2,q)$ is the group of equivalence classes $[A]_\sim$ of invertible $2 \times 2$ matrices $A$ defined over the field $\mathbb{F}_q$ where $A \sim B$ if there exists $\lambda \in \mathbb{F}_q^\times$ such that $A=\lambda B$ and where $[A]_\sim [B]_\sim := [AB]_\sim$. Finally, we recall that a $t-(n,r,\lambda)$ design $\calH$ is an $r$-uniform hypergraph defined over a vertex set of size $n$ with the property that any set of $t$ vertices appear together in exactly $\lambda$ hyperedges of $\calH$.

\begin{theorem}\label{thm:main-const}
For each prime power $q \equiv 3 \pmod 4$ there exists a $3-(q+1,4,\frac{q+1}{4})$ design denoted $\calH_q$ on $q+1$ vertices with the following properties:
\begin{itemize}
\item[(a)] any set of $5$ vertices spans 0 or 2 edges;
\item[(b)] $\calH_q$ has $\frac{q+1}{16}\binom{q+1}{3}$ hyperedges; and
\item[(c)] $\calH_q$ is $3$-transitive with a subgroup of its automorphism group isomorphic to $\PGL(2,q)$.
\end{itemize}
\end{theorem}

A straightforward modification of the proof of an upper bound on Tur\'{a}n numbers due to de Caen \cite{dC83} shows that any $4$-uniform hypergraph on $n$ vertices in which every $5$ vertices span either $0$ or $2$ hyperedges has at most $\frac{n}{16}\binom{n}{3}$ hyperedges (see Section \ref{sec:ub}).  Furthermore, any such hypergraph attaining the maximum number of hyperedges is a $3$-design.  Designs arising from both the groups $\PGL(2, q)$ and $\PSL(2, q)$ (the subgroup of $\PGL(2,q)$ consisting of classes of matrices with determinant 1) have been examined in a number of previous works, including papers by Cusack, Graham, and Kreher \cite{CGK95}, by Cameron, Omidi, and Tayfeh-Rezaie \cite{COT-R06}, and by Cameron, Maimani, Omidi, and Tayfeh-Rezaie \cite{CMOT-R06}.  To the best of our knowledge, the hypergraph in Theorem \ref{thm:main-const} has not been previously studied.

The proofs of the upper bounds on the number of hyperedges in the class of hypergraphs of interest are given, for completeness, in Section \ref{sec:ub}.  This shows that the hypergraphs constructed in Theorem \ref{thm:main-const} are extremal.  Furthermore, the same upper bound on the number of hyperedges applies for hypergraphs in which every set of $5$ vertices contains \emph{at most} $2$ hyperedges.  Thus, the upper bound, together with Theorem \ref{thm:main-const} imply the following result on Tur\'{a}n numbers for the $4$-uniform hypergraph on $5$ vertices with $3$ hyperedges, which is unique up to isomorphism.

\begin{corollary}\label{cor:ex-num}
For any prime power $q \equiv 3\pmod{4}$,
\[
\ex(q+1, \{1234, 1235, 1245\}) = \frac{(q+1)}{16} \binom{q+1}{3}.
\]
\end{corollary}

For fixed $n$, $\{1234, 1235, 1245\}$-free $4$-uniform hypergraphs on $n$ vertices with exactly $\ex(n, \{1234, 1235, 1245\})$ hyperedges need not be unique.  Hughes \cite{dH65} considered certain designs arising from groups and showed that there is a $3-(12, 4, 3)$ design with $165$ blocks associated with the Mathieu group $\M_{11}$ which has $\M_{11}$ as a group of automorphisms.  In Section \ref{sec:M11}, this design is examined and the following properties are shown.

\begin{theorem}\label{thm:M11}
There exists a $3-(12,4,3)$ design $\calM$ on $12$ vertices with the following properties:
\begin{itemize}
\item[(a)] any set of $5$ vertices spans 0 or 2 edges;
\item[(b)] $\calM$ has $165$ hyperedges;
\item[(c)] $\calM$ is 3-transitive with a subgroup of its automorphisms isomorphic to $\M_{11}$;
\end{itemize}
\end{theorem} 

The hyperedges of the hypergraph $\mathcal{M}$ are listed in the Appendix for the interested reader.  One can show that the hypergraph $\mathcal{M}$ is not isomorphic to the hypergraph $\mathcal{H}_{11}$.  In particular, using SAGE \cite{wS15}, it was verified by direct examination that $\mathcal{H}_{11}$ has the property that every set of $6$ vertices contains at least one hyperedge and that $\mathcal{M}$ contains $22$ independent sets of size $6$.

The extremal numbers given in Corollary \ref{cor:ex-num} give a new proof of the fact that the Tur\'{a}n density for the $5$ vertex hypergraph with $3$ edges is
\begin{equation}\label{eq:density}
	\pi(\{1234, 1235, 1245\}) = \frac{1}{4}.
\end{equation} The first proof for the Tur\'{a}n density in Equation \eqref{eq:density} was given by Baber \cite{rB} who proved the lower bound using the following construction based on random tournaments.  Let $T$ be any random tournament on $n$ vertices and construct a $4$-uniform hypergraph $\mathcal{H}_T$ by taking all sets of four vertices $\{a, b, c, d\}$ with the property that the edges are directed $a \to b$, $b \to c$ and $c \to a$ and such that the edges between $d$ and $a$, $b$, or $c$ are either all directed towards $d$ or all directed away from $d$.  The two possible sub-tournaments that give hyperedges are shown in Figure \ref{fig:tourn}. One can check that $\mathcal{H}_T$ has the property that any $5$ vertices contain at most two hyperedges and that, in fact, every $5$ vertices span $0$ or $2$ hyperedges.

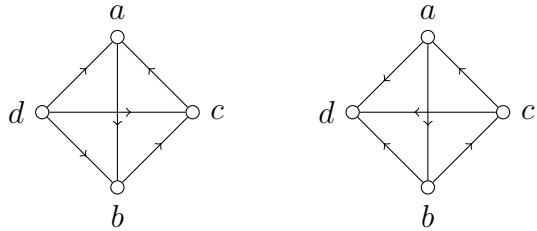
\begin{figure}[htb]
\begin{center}
\begin{tikzpicture}
	[decoration={markings, mark=at position 0.6 with {\arrow{>}}}] 
	\tikzstyle{vertex}=[circle, draw=black,  minimum size=5pt,inner sep=0pt]
	
	\node[vertex, label=left:{$d$}] at (-1, 0) (x) {};
	\node[vertex, label=above:{$a$}] at (0, 1) (y) {};
	\node[vertex, label=below:{$b$}] at (0, -1) (z) {};
	\node[vertex, label=right:{$c$}] at (1, 0) (w) {};
	
	\draw[postaction={decorate}] (y) -- (z);
	\draw[postaction={decorate}] (z) -- (w);
	\draw[postaction={decorate}] (w) -- (y);
	\draw[postaction={decorate}] (x) -- (y);
	\draw[postaction={decorate}] (x) -- (z);
	\draw[postaction={decorate}] (x) -- (w);
\end{tikzpicture} \hspace{20pt}
\begin{tikzpicture}
	[decoration={markings, mark=at position 0.6 with {\arrow{>}}}] 
	\tikzstyle{vertex}=[circle, draw=black,  minimum size=5pt,inner sep=0pt]
	
	\node[vertex, label=left:{$d$}] at (-1, 0) (x) {};
	\node[vertex, label=above:{$a$}] at (0, 1) (y) {};
	\node[vertex, label=below:{$b$}] at (0, -1) (z) {};
	\node[vertex, label=right:{$c$}] at (1, 0) (w) {};
	
	\draw[postaction={decorate}] (y) -- (z);
	\draw[postaction={decorate}] (z) -- (w);
	\draw[postaction={decorate}] (w) -- (y);
	\draw[postaction={decorate}] (y) -- (x);
	\draw[postaction={decorate}] (z) -- (x);
	\draw[postaction={decorate}] (w) -- (x);
\end{tikzpicture} 
\end{center}
\caption{Possible edge directions for hyperedge}
\label{fig:tourn}
\end{figure}

The expected number of such hyperedges given by a random tournament on $n$ vertices is $\frac{1}{4}\binom{n}{4}$ and so for every $n$, there is such a tournament yielding a hypergraph with at least this many hyperedges.

In fact, since the construction yields hypergraphs with every $5$ vertices spanning exactly $0$ or $2$ hyperedges, it also shows that
\[
\limsup_{n \to \infty} \left( \frac{\max\{e(\mathcal{H}) \mid |V(\mathcal{H})| = n \text{ and every $5$-set spans $0$ or $2$ edges}\}}{\binom{n}{4}} \right)= \frac{1}{4}.
\]

The Paley tournaments are a collection of tournaments that exhibit some pseudo-random properties (see, for example~\cite{nA95}).  In Section \ref{sec:tourn}, we give a construction starting from a Paley tournament showing that the hypergraphs $\mathcal{H}_q$ from Theorem \ref{thm:main-const} can be represented by tournaments using Baber's construction.  Furthermore, by examining a `switching' operation on tournaments that preserves the resulting hypergraphs, we give easily verifiable necessary conditions for a $4$-uniform hypergraph to be represented in this form.  In particular, we show that a hypergraph $\calH$ with the property that any set of $5$ vertices spans 0 or 2 edges cannot be realized using a tournament whenever it contains a pair of vertices $u,v$ such that $N_\calH(u,v)$ contains an odd cycle (see Proposition \ref{p:oddcycle}).  This is used to show that not only is the hypergraph $\mathcal{M}$ from Theorem \ref{thm:M11} not equal to $\mathcal{H}_{11}$, but that it also cannot be represented as $\mathcal{H}_T$ for any tournament on $12$ vertices.
 
The structure of the remainder of the paper is as follows.  In Section \ref{sec:design}, the construction of the family of $4$-uniform hypergraphs described in Theorem \ref{thm:main-const} is given.  In Lemma \ref{lem:3tran}, it is shown that the hypergraphs are $3$-transitive and that there is a subgroup of automorphisms isomorphic to $\PGL(2, q)$.  In Theorem \ref{thm:paley0or2}, it is shown that for each of the constructed hypergraphs, every set of $5$ vertices spans either $0$ or $2$ hyperedges.  In Theorem \ref{thm:paley-edge-count}, it is shown that the hypergraphs $\mathcal{H}_q$ are $3$-designs and it is shown that for every $q$, the hypergraph $\mathcal{H}_q$ has exactly $\frac{(q+1)}{16}\binom{q+1}{3}$ hyperedges. 

In Section \ref{sec:M11}, the properties of the hypergraph $\mathcal{M}$ described in Theorem \ref{thm:M11} are given. 

In Section \ref{sec:ub}, the modification of de Caen's counting argument is given to prove an upper bound on the number of hyperedges for $r$-uniform hypergraphs in which any set of $r+1$ vertices spans either exactly $0$ or $2$ hyperedges or else at most $2$ hyperedges. 

In Section \ref{sec:tourn}, Baber's construction for hypergraphs from tournaments is examined and it is shown in Theorem \ref{thm:isotourn} that for every $q$, there is a tournament $T^*(q)$ for which $\mathcal{H}_q = \mathcal{H}_{T^*(q)}$.  We formally define what it means for two tournaments $T_1$ and $T_2$ to be switching equivalent in Definition \ref{def:switch}. Some implications of this switching operation are given for the hypergraphs resulting from tournaments and these are used to show that the hypergraph $\mathcal{M}$ does not arise from a tournament.

Finally, in Section \ref{sec:open}, we note some remaining open problems and further possible directions.

\section{Paley hypergraphs}\label{sec:design}

Throughout, let $q=p^\ell$ be a prime power with $q \equiv 3 \pmod 4$. The purpose of this section is to construct a 4-uniform hypergraph $\mathcal{H}_q$ on $q+1$ vertices with $\frac{(q+1)}{16} \binom{q+1}{3}$ edges and with the property that every set of 5 vertices spans exactly 0 or exactly 2 hyperedges.  This construction uses the projective line over a finite field.

The hypergraphs given by the construction in Definition~\ref{def:paley-hgraph} are called Paley $4$-graphs because of the hyperedges are defined in terms of the quadratic residue properties of their vertices and because of the connection to the usual Paley tournaments that is described in Section~\ref{sec:tourn}.  Several other generalizations of the Paley graph construction to hypergraphs have been given previously.  Chung and Graham~\cite{CG90} define $k$-graphs on vertex set $\{1, 2, \ldots, p\}$ whose hyperedges are those $k$-sets whose sum is a quadratic residue modulo $p$ as an example of a `quasi-random hypergraph'.  Alon~\cite{nA90} gave a construction for $k$-uniform hypergraphs in terms of quadratic non-residues.  Poto\v{c}nik and \v{S}ajna~\cite{PS09} gave a construction for $k$-uniform hypergraphs on $n$ vertices for any prime power $n$ in which edges are defined by the coset in the multiplicative group, $\mathbb{F}_n^*$, of the van der Monde determinant of the $k$-sets.  In the case where $k = 4$, this construction requires that $n \equiv 1 \pmod{8}$.  The hypergraphs constructed in this section are, in general, different from any of these others given previously in the literature.

\begin{definition}
Define an equivalence relation, $\sim$, on $\mathbb{F}_q^2\setminus\{(0,0)\}$ by $(a, b) \sim (c, d)$ if{f} there exists $\lambda \in \mathbb{F}_q\setminus\{0\}$ with $(a, b) = (\lambda c, \lambda d)$.  The \emph{projective line $\mathbb{P}^1\mathbb{F}_q$} is the set of equivalence classes of $\mathbb{F}_q^2 \setminus \{(0,0)\}$ under this equivalence relation. We write $[a:b]$ for the equivalence class of the point $(a, b)$ so that:
\[
\mathbb{P}^1\mathbb{F}_q = \{[x:1] \mid x \in \mathbb{F}_q\} \cup \{[1:0]\}.
\]
\end{definition}

\begin{definition}
Define $D: \mathbb{F}_q^2 \times \mathbb{F}_q^2 \to \mathbb{F}_q$ by
\[
D\left((u_1, u_2), (v_1, v_2) \right)  = \begin{vmatrix} u_1 & v_1\\ u_2 & v_2\end{vmatrix}= u_1v_2 - u_2 v_1,
\]
called the \emph{determinant of $(u_1, u_2)$ and $(v_1, v_2)$}.
\end{definition}

Note that the determinant is not constant on the equivalence classes that define the projective line. Next, let $\chi: \mathbb{F}_q \to \{-1, 0, +1\}$, be the square character of the multiplicative group $\mathbb{F}_q$. Explicitly,
\[
\chi(x) = \begin{cases}
		0 	&\text{if } x=0\\
		+1	&\text{if $x$ is a square in $\mathbb{F}_q$, and}\\
		-1	&\text{otherwise.}
	\end{cases}
\]
Since $\chi$ is a linear character, it is multiplicative. Also note that $\chi(-1) = -1$ since $q \equiv 3 \pmod{4}$.

\begin{definition}\label{def:S}
Define the function $S: (\mathbb{F}_q^2)^4 \to \{-1, 0, 1\}$ by
\[
S(a, b, c, d) = \chi\left(D(a, b) D(b, c) D(c, d) D(d, a)\right)
\]
As the function $S$ is constant over equivalence classes of $\mathbb{F}_q^2$, it extends in the natural way to $(\mathbb{P}^1 \mathbb{F}_q)^4$.
\end{definition}

With this definition, the key construction in this paper can now be given.

\begin{definition}[Paley $4$-graph]\label{def:paley-hgraph}
Define $\mathcal{H}_q$ to be the $4$-uniform hypergraph on vertex set $\mathbb{P}^1\mathbb{F}_q$ where $\{a, b, c, d\}$ is a hyperedge if{f} for every permutation $\pi$ of  $\{a, b, c, d\}$,
\[
	S(\pi(a), \pi(b), \pi(c), \pi(d)) = -1.
\]
\end{definition}

The following lemma gives a useful criterion for testing whether or not a given set of 4 vertices lies in $\mathcal{H}_q$.

\begin{lemma}\label{lem:relabel}
For every four distinct points $a, b, c, d \in \mathbb{P}^1\mathbb{F}_q$, the set $\{a, b, c, d\}$ is a hyperedge in $\mathcal{H}_q$ if{f}
\[
S(a, b, c, d) = S(a, b, d, c) = S(a, c, b, d) = -1
\]
and $\{a, b, c, d\} \notin \mathcal{H}_q$ if{f} among $S(a, b, c, d), S(a, b, d, c)$, and $S(a, c, b, d)$ exactly one is $-1$ and the other two are $+1$.
\end{lemma}

\begin{proof}
Since $S$ is invariant under cyclic permutations of its variables $S(a, b, c, d), S(a, b, d, c)$ and $S(a, c, b, d)$ determine the values of $S$ under all permutations of the set $\{a, b, c, d\}$.  This implies the first part of the lemma.

For the second part, note that
\begin{align}
S(a, b, c, d)&S(a, b, d, c) \notag\\
	&=\chi\left(D(a, b)D(b, c) D(c, d) D(d, a)\right) \chi\left(D(a, b) D(b, d) D(d, c) D(c, a) \right) \notag\\
	&=\chi\left(D(b, c) D(c, d) D(d, a) D(b, d) D(d, c) D(c, a)\right) \notag\\
	&=-\chi\left(D(a, c) D(c, b) D(b, d) D(d, a)\right) \notag\\
	&=-S(a, c, b, d). \label{eq:threeS}
\end{align}
Since these are all either $+1$ or $-1$, then equation~\eqref{eq:threeS} shows that if one of $S(a, b, c, d), S(a, b, d, c)$ and $S(a, c, b, d)$ are $+1$, then exactly two are $+1$ and the other is $-1$.
\end{proof}

We next show that $\mathcal{H}_q$ is 3-transitive. Recall from the introduction that $\PGL(2,q)$ is the set of all invertible $2 \times 2$ matrices over $\mathbb{F}_q$ equivalent up to multiplication by scalar matrices. Observe that $\PGL(2,q)$ acts naturally on $\mathbb{P}^1\mathbb{F}_q$ as follows:

 $$A = \begin{pmatrix} a &b \\ c &d \end{pmatrix} \mbox{ sends } [u_1, u_2] \mbox{ to } [au_1+bu_2: cu_1+du_2].$$ It is well known (see for example \cite[p. 245]{DM}) that this action is 3-transitive. 

\begin{lemma}\label{lem:3tran}
For any prime power $q \equiv 3 \pmod{4}$, the hypergraph $\mathcal{H}_q$ is $3$-transitive.
\end{lemma}

\begin{proof}
Let $u:= (u_1, u_2)$ and $v:= (v_1, v_2)$ be elements of $\mathbb{F}_q^2$ and let $A \in \PGL(2,q)$. Then,
\begin{align*}
D(uA^T, vA^T)
	& = D\left( (au_1+b u_2, cu_1 + du_2), (a v_1 + b v_2, c v_1+ d v_2)\right)\\
	&=(au_1 + bu_2)(cv_1+dv_2) - (av_1 + b v_2)(c u_1 + d u_2)\\
	&=(ad-bc)(u_1v_2 - v_1 u_2)\\
	&=\det(A) D(u, v).
\end{align*}

Thus, for any $x, y, u, v \in \mathbb{F}_q^2$,
\[
S(xA^T, yA^T, uA^T, vA^T) = \chi\left(\det(A)^4\right) S(x, y, u, v) = S(x, y, u, v).
\]
This shows that the map $x \mapsto xA^T$ is an automorphism of the hypergraph $\mathcal{H}_q$. In particular, $\PGL(2,q) \leq$ Aut$(\mathcal{H}_q)$ and $\mathcal{H}_q$ is 3-transitive.

\end{proof}

We can now prove our first main result:

\begin{theorem}\label{thm:paley0or2}
For any prime power $q \equiv 3\pmod{4}$ and five distinct vertices $a, b, c, d, e \in \mathbb{P}^1\mathbb{F}_q$, in $\mathcal{H}_q$, the set $\{a, b, c, d, e\}$ either contains $0$ or $2$ hyperedges.
\end{theorem}

\begin{proof}
If there are no edges in $\{a, b, c, d, e\}$, then we are done, so suppose without loss of generality that $\{a, b, c, d\} \in E(\mathcal{H}_q)$.  Note that as long as $a, b, c, d, e$ are all distinct, using the definition of the function $S$ (Definition \ref{def:S}) and the multiplicativity of $\chi$ we obtain
\begin{equation}\label{eq:doublecover}
S(a, b, c, d) S(b, c, e, d) S(a, d, b, e) S(a, e, d, c) S(a, b, e, c) = +1.
\end{equation}
This implies that not all $5$ different $4$-sets in $\{a, b, c, d, e\}$ are edges or else the product in equation~\eqref{eq:doublecover} would be $-1$.  Therefore, there is at least one non-edge,  say $\{b, c, d, e\}$.  By Lemma \ref{lem:relabel}, after possibly relabelling points, we can assume that $S(b, c, d, e) = -1$ and $S(b, c, e, d) = S(b, d, c, e) = +1$. Furthermore, by Lemma \ref{lem:3tran}, we may assume that $b = [0:1]$, $c = [1:1]$ and $d = [1:0]$.  Thus, it can be assumed that $a = [a_1:1]$ and $e = [e_1:1]$ for some $a_1,e_1 \in \mathbb{F}_q \backslash \{0,1\}$.  Note that for any $x \in \mathbb{F}_q$, $D((x, 1), (1, 0)) = -1$ and $D((1, 0), (x, 1)) = 1$.

Then, $1-e_1$ is a square in $\mathbb{F}_q$ since
\[
+1 = S([0:1], [1:1], [e_1:1], [1:0]) =\chi\left((-1)\cdot(1-e_1) \cdot (-1)\cdot (1) \right) = \chi\left(1-e_1 \right).
\]
Also, $e_1$ is a non-square since $S([0:1], [1:1], [1:0], [e_1:1]) = -1.$  The fact that 
\[
S([0:1], [1:1], [1:0], [a_1:1]) = S([0:1], [1:1], [a_1:1], [1:0]) = -1
\]
 implies that both $a_1$ and $1-a_1$ are non-squares.

Thus, since 
\[
S([a_1:1], [1:0], [e_1:1], [0:1]) = \chi\left((-1)(1)(e_1)(-a_1)\right) = +1
\]
we must have that $\{a, b, d, e\}$ is not an edge.

Finally, note that
\begin{align*}
S([a_1:1], [0:1], [e_1:1], [1:1]) &= \chi\left(a_1 (-e_1)(e_1-1)(1-a_1)\right) = -1,\\
S([a_1:1], [0:1], [1:1],  [e_1:1]) & = \chi\left(a_1(-1)(1-e_1)(e_1-a_1) \right) = \chi\left(e_1-a_1\right),\\
S([a_1:1], [1:0], [e_1:1], [1:1]) &= \chi\left((-1)(+1)(e_1-1)(1-a_1) \right) = -1, \text{ and}\\
S([a_1:1], [1:0], [1:1],  [e_1:1]) &= \chi\left((-1)(+1)(1-e_1)(e_1 - a_1) \right) = -\chi\left(e_1-a_1\right).
\end{align*}
Hence if $e_1-a_1$ is a square in $\mathbb{F}_p$, then $\{a, b, c, e\}$ is a non-edge and $\{a, c, d, e\}$ is an edge.  Conversely, if $e_1 - a_1$ is a non-square, then $\{a, b, c, e\}$ is an edge and $\{a, c, d, e\}$ is a non-edge.

In either case, if the set $\{a, b, c, d, e\}$ contains at least one hyperedge, then it contains exactly two and the proof is complete.
\end{proof}

Next, we count the number of edges in $\mathcal{H}_q$. For this, two facts about sums of the function $\chi$ are used.  The only condition required for each is that $q$ is an odd prime power.  The first identity is
\begin{equation}\label{eq:sumLeg}
\sum_{x \in \mathbb{F}_q} \chi(x) = 0
\end{equation}
and the second is that for any $y \neq 0$,
\begin{equation}\label{eq:sumLegpairs}
\sum_{x \in \mathbb{F}_q} \chi(x)\chi(x+y) = -1.
\end{equation}
Both of these facts are standard exercises in number theory. We remark that (\ref{eq:sumLegpairs}) is the convolution $\chi * \chi$ in the sense of Fourier analysis on finite groups.

\begin{theorem}\label{thm:paley-edge-count}
For any prime power $q \equiv 3 \pmod{4}$, the hypergraph $\mathcal{H}_q$ has $e(\mathcal{H}_q) = \frac{(q+1)}{16} \binom{q+1}{3}$ and every set of $3$ vertices occurs in exactly $(q+1)/4$ hyperedges. 
\end{theorem}

\begin{proof}
Consider first hyperedges of the form $\{[0:1], [1:1], [1:0], [a:1]\}$ with $a \neq 0,1$.  If such a set is a hyperedge then by Lemma \ref{lem:relabel},
\begin{align*}
-1 = S([a:1], [0:1], [1:0], [1:1]) &= \chi\left(a\right)(-1)(+1)\chi\left(1-a\right) =\chi \left(a\right)\chi\left(a-1\right); \\
-1 = S([a:1], [0:1], [1:1], [1:0]) &= \chi\left(a\right)(-1)(-1)(+1) =\chi \left(a\right); \mbox{ and}\\
-1 = S([a:1],  [1:1], [0:1], [1:0]) &= \chi\left(a-1 \right) (+1)(-1)(+1)= -\chi\left(a-1 \right).
\end{align*}
Note that any two of these equations implies the third. Thus, $[a:1]$ is in a hyperedge with $\{[0:1], [1:1], [1:0]\}$ if{f} both $a$ and $1-a$ are non-squares in $\mathbb{F}_q$.  Consider
\[
\frac{1}{4}\left(1 - \chi\left(a \right) \right)\left(1 - \chi\left(1-a \right) \right) = 
	\begin{cases}
		1	&\text{if $a$ and $1-a$ are both non-square, and}\\
		0	&\text{otherwise}.
	\end{cases}
\]

The number of hyperedges in $\mathcal{H}_q$ of containing $\{[0:1], [1:1], [1:0]\}$ is then exactly
\begin{multline}\label{eq:edge-count}
\frac{1}{4} \sum_{a \in \mathbb{F}_q \setminus \{0,1\}} \left(1 - \chi\left(a\right)\right)\left(1 - \chi\left(1-a\right)\right) \\
	= \frac{1}{4}\sum_{a \in \mathbb{F}_q \setminus \{0,1\}}\left(1 - \chi\left(a \right) - \chi\left(1-a\right) + \chi\left(a \right)\chi\left(1-a\right)\right) .
\end{multline}

Consider the terms in equation \eqref{eq:edge-count} separately.  Note that, by equation \eqref{eq:sumLeg},
\[
\sum_{a \in \mathbb{F}_q \setminus \{0,1\}} \chi\left(a \right) = -\chi\left(1\right) = -1, \qquad \text{and} \qquad 
	\sum_{a \in \mathbb{F}_q \setminus \{0,1\}}\chi\left(1-a\right) = -\chi\left(1\right) = -1.
\]
Then, by equation \eqref{eq:sumLegpairs},
\[
\sum_{a \in \mathbb{F}_q \setminus \{0,1\}}\chi\left(a\right)\chi\left(1-a\right)
	=-\sum_{a \in \mathbb{F}_q \setminus \{0,1\}}\chi\left(a\right)\chi\left(a-1\right)
	=-(-1 - 0 - 0) = 1.
\]

Thus, substituting into equation \eqref{eq:edge-count} gives,
\[
\frac{1}{4} \sum_{a \in \mathbb{F}_q \setminus \{0,1\}} \left(1 - \chi\left(a\right)\right)\left(1 - \chi\left(1-a\right)\right) 
	 = \frac{1}{4}\left((q-2) - (-1) - (-1) +1 \right) = \frac{q+1}{4}.
\]

Thus, the three vertices $\{[0:1], [1:1], [1:0]\}$ are contained together in exactly $(q+1)/4$ hyperedges of $\mathcal{H}_p$ and since the hypergraph is $3$-transitive, the same is true of any other $3$ vertices.  That the total number of hyperedges is
\[
e(\mathcal{H}_q) = |E(\mathcal{H}_q)| = \frac{(q+1)}{16}\binom{q+1}{3}
\]
follows immediately.
\end{proof}

\section{A $4$-hypergraph associated to $\M_{11}$}\label{sec:M11}

In this section, we describe a single $4$-uniform hypergraph on $12$ vertices with the property that any $5$ vertices span either $0$ or $2$ hyperedges and that has the same number of hyperedges as $\mathcal{H}_{11}$ (another hypergraph on $12$ vertices).  

Hughes \cite{dH65} examined certain designs arising from groups and showed that there is a $3-(12, 4, 3)$ design which occurs as an orbit on 4-subsets under the natural action of the Mathieu group $\M_{11}$ on 12 points. We denote this design by $\mathcal{M}$ here.  A listing of the hyperedges of $\mathcal{M}$ is given in an appendix.  One can verify directly that the hypergraph $\mathcal{M}$, with $165$ hyperedges, has the property that every set of $3$ vertices is contained in exactly $3$ hyperedges and that every set of $5$ vertices contains either exactly $0$ or $2$ hyperedges.  

This $3$-design was also examined by Devillers, Giudici, Li, and Praeger \cite{DGHP} and their results can be used to give alternative proofs of these facts.  In \cite{DGHP}, a graph $\Gamma$ is defined with vertex set being the hyperedges of $\mathcal{M}$ and two vertices $A, B$ being adjacent if{f} $|A \cap B| = 3$.

\begin{theorem}[in Theorem 2.5 of \cite{DGHP}]\label{thm:gamma-graph}
The graph $\Gamma$ is an $8$-regular graph on $165$ vertices with the property that any two cliques of size $3$ intersect in at most one vertex.
\end{theorem}

Theorem \ref{thm:gamma-graph} is used to give another proof of the fact that any set of $5$ vertices of $\mathcal{M}$ contains either $0$ or exactly $2$ hyperedges.

\begin{proposition}\label{prop:m11-0or2}
Let $a, b, c, d, e$ be vertices of $\mathcal{M}$ with $\{a, b, c, d\} \in \mathcal{M}$.  There is exactly one other hyperedge of $\mathcal{M}$ in the set $\{a, b, c, d, e\}$.
\end{proposition}

\begin{proof}
Note that since $\mathcal{M}$ is a $3-(12, 4, 3)$ design, then for each subset of $X \subseteq \{a, b, c, d\}$, of size $3$, there are two vertices $x, y \in V(\mathcal{M})$ so that $X \cup \{x\}, X \cup \{y\} \in \mathcal{M}$.  Then the three sets $\{a, b, c, d\}$, $X \cup \{x\}$, and $X \cup \{y\}$ correspond to a clique of size $3$ in $\Gamma$.  Since these four cliques intersect in the vertex $\{a, b, c, d\}$, they are otherwise pairwise disjoint.  Let $x_1, x_2, x_3, x_4, x_5, x_6, x_7, x_8$ be such that $\mathcal{M}$ contains the hyperedges
\[
\begin{tabular}{llll}
	$\{a, b, c, x_1\}$, 	&$\{a, b, d, x_3\}$, 	&$\{a, c, d, x_5\}$,	&$\{b, c, d, x_7\}$\\
	$\{a, b, c, x_2\}$, 	&$\{a, b, d, x_4\}$, 	&$\{a, c, d, x_6\}$,	&$\{b, c, d, x_8\}$.
\end{tabular}
\]
Since these sets are all distinct, $x_1 \neq x_2$, $x_3 \neq x_4$, $x_5 \neq x_6$, and $x_7 \neq x_8$.  Furthermore, since $\Gamma$ is $8$-regular, these are the only $8$ neighbours of the vertex $\{a, b, c, d\}$, as in Figure \ref{fig:Gamma-graph}.

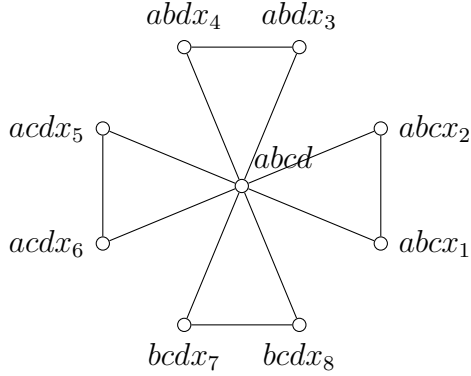
\begin{figure}[htb]
\begin{center}
\begin{tikzpicture}
	\tikzstyle{vertex}=[circle, draw=black,  minimum size=5pt,inner sep=0pt]
	
	\node[vertex, label={80:$abcd$}] at (0, 0) (0) {};
	\node[vertex, label=right:{$abcx_1$}] at (337.5:2) (1) {};
	\node[vertex, label=right:{$abcx_2$}] at (22.5:2) (2) {};
	\node[vertex, label=above:{$abdx_3$}] at (67.5:2) (3) {};
	\node[vertex, label=above:{$abdx_4$}] at (112.5:2) (4) {};
	\node[vertex, label=left:{$acdx_5$}] at (157.5:2) (5) {};
	\node[vertex, label=left:{$acdx_6$}] at (202.5:2) (6) {};
	\node[vertex, label=below:{$bcdx_7$}] at (247.5:2) (7) {};
	\node[vertex, label=below:{$bcdx_8$}] at (292.5:2) (8) {};
	\draw (0) -- (1) -- (2) --(0)-- (3) -- (4) -- (0) -- (5) -- (6) -- (0) -- (7)--(8)--(0); 
\end{tikzpicture}
\end{center}
\caption{Neighbourhood of a vertex in $\Gamma$}
\label{fig:Gamma-graph}
\end{figure}

Suppose, for some $i \neq j$, that $x_i = x_j$ (with $j \neq i+1$).  Let $A, B \subseteq \{a, b, c, d\}$ be the two different sets of size $3$ with $A \cup \{x_i\}, B \cup \{x_i\} \in \mathcal{M}$.  Then, since $|A \cap B| = 2$, there is an edge in $\Gamma$ between $A \cup \{x_i\}$ and $B \cup \{x_i\}$.  As shown previously, there is $x_{i'} \neq x_i$ with $A \cup \{x_{i'}\} \in \mathcal{M}$.  Then, in $\Gamma$, the three vertices $\{a, b, c, d\}$, $A \cup \{x_i\}$ and $B \cup \{x_i\}$ form a clique on $3$ vertices that shares two vertices with the clique formed by $\{a,b, c, d\}$, $A \cup \{x_i\}$ and $A \cup \{x_{i'}\}$.  This contradicts Theorem \ref{thm:gamma-graph} and so each of $x_1, x_2, \ldots, x_8$ are distinct and different from $a, b, c$ or $d$.  In particular, there exactly one $i \in [1, 8]$ with $x_i = e$.  Thus, there is exactly one hyperedge containing $e$ and three vertices from $\{a, b, c, d\}$.  

This completes the proof that every set of $5$ vertices in $\mathcal{M}$ either contains no hyperedges or contains exactly $2$. 
\end{proof}

Furthermore, using a connection to Witt design $\mathcal{W}_{11}$, it is shown in \cite{DGHP} that the full automorphism group of $\mathcal{M}$ is the group $\M_{11}$.

As described in the introduction, a direct examination shows that $\mathcal{M}$ is not isomorphic to $\mathcal{H}_{11}$, although they both have $165$ hyperedges.  In Section \ref{sec:tourn} to come, another property of hypergraphs is examined to prove that the two hypergraphs $\mathcal{M}$ and $\mathcal{H}_{11}$ are different.

\section{Extremal bounds}\label{sec:ub}

We now consider the maximum number of hyperedges possible in a $4$-uniform hypergraph with the property that every $5$ vertices span either $0$ or $2$ hyperedges, and deduce that the hypergraph $\mathcal{H}_q$ considered in Section 1 is maximal among 4-hypergraphs on $q+1$ vertices with this property. The bound given in Proposition \ref{prop:ub0or2} follows directly from an argument used by de Caen \cite{dC83} to give upper bounds for the Tur\'{a}n numbers for complete hypergraphs.  The full proof is included here both for completeness and to highlight the fact that those hypergraphs which attain the upper bound are necessarily designs.  Though we shall only use this result in the case when $r=4$, we state and prove the result for arbitrary $r$-uniform hypergraphs. 

\begin{proposition}\label{prop:ub0or2}
Let $r \geq 2$ and let $\mathcal{H}$ be an $r$-uniform hypergraph on $n$ vertices with the property that every set of $r+1$ vertices contains at most $2$ hyperedges.  Then,
\[
|E(\mathcal{H})| \leq \frac{n}{r^2} \binom{n}{r-1},
\]
with equality if{f} $\mathcal{H}$ is such that every set of $(r-1)$ vertices occurs in exactly $n/r$ hyperedges.
\end{proposition}

\begin{proof}
 The proof involves double-counting the set
 \begin{equation}\label{eq:edge-nonedge}
  \left\{(A, B) \mid |A| = |B|=r,\ |A \cap B| = r-1,\ A \in \mathcal{H} \text{ and } B \notin \mathcal{H} \right\}.
 \end{equation}
 
 The size of the set in \eqref{eq:edge-nonedge} can be bounded from below as follows.  Let $E_1, E_2, \ldots, E_m$ be the hyperedges of $\mathcal{H}$.  Fix $i \leq m$ and $x \notin E_i$. Since the set $E_i \cup \{x\}$ contains $r+1$ vertices and at least one hyperedge, then by assumption, it contains at most $2$.  That is, there is at most one vertex $y \in E_i$  so that $E_i \cup \{x\} \setminus \{y\} \in \mathcal{H}$.  That is, for each $z \in E_i \setminus \{y\}$, the pair $(E_i, E_i \cup \{x\} \setminus \{z\})$ is in the set in \eqref{eq:edge-nonedge}.  Furthermore, all pairs in this set are of this form.  Thus,
 \begin{align}
|\big\{(A, B) \mid &|A| = |B|=r,\ |A \cap B| = r-1,\ A \in \mathcal{H} \text{ and } B \notin \mathcal{H} \big\} \notag\\
	&=\sum_{i = 1}^{m} \sum_{x \notin E_i} |\{z \in E_i \mid E_i \cup\{x\} \setminus \{z\} \notin \mathcal{H}\}|\notag\\
	&\geq \sum_{i=1}^m \sum_{x \notin E_i} (r-1) \notag\\
	&=m(n-r)(r-1) = e(\mathcal{H}) (n-r)(r-1). \label{eq:exact-count}
\end{align}

Note that the inequality in \eqref{eq:exact-count} is, in fact, an identity in the case that every $r+1$ vertices span either $0$ or exactly $2$ hyperedges.
 
 For an upper bound on the size of the set in equation \eqref{eq:edge-nonedge}, order the $(r-1)$-sets of vertices $\{C_i \mid 1\leq i \leq \binom{n}{r-1}\}$ and for each $i \leq \binom{n}{r-1}$, let $a_i$ be the number of hyperedges of $\mathcal{H}$ containing the set $C_i$.  Note that, by double counting, $\sum a_i = r e(\mathcal{H})$.  Further, the number of pairs $(A, B)$ with $A \in \mathcal{H}$, $B \notin \mathcal{H}$ and $A \cap B = C_i$ is $a_i(n-r+1-a_i)$.  Thus,
 \begin{align}
 \sum_{i = 1}^{\binom{n}{r-1}} &\left| \{(A, B) \mid A \in \mathcal{H},\ |B| = r, B \notin \mathcal{H}, A \cap B = C_i \}\right| \notag\\
 	&=\sum_{i = 1}^{\binom{n}{r-1}} a_i (n-r+1 - a_i) \notag\\
	&=(n-r+1) r e(\mathcal{H}) - \sum_{i = 1}^{\binom{n}{r-1}} a_i^2 \notag\\
	&\leq (n-r+1) r e(\mathcal{H}) - \frac{1}{\binom{n}{r-1}} \left(r e(\mathcal{H}) \right)^2 &&\text{(by Jensen's ineq.)} \label{eq:conv}\\
	& = (n-r+1)r e(\mathcal{H}) - \frac{r^2}{\binom{n}{r-1}} e(\mathcal{H})^2. \label{eq:ub}
 \end{align}
Further, by convexity, equality holds in \eqref{eq:conv} if{f} all of the $a_i$s are equal.

Combining equations \eqref{eq:exact-count} and \eqref{eq:ub} shows that
\begin{equation}\label{eq:combined-bd}
e(\mathcal{H}) \leq \frac{n}{r^2} \binom{n}{r-1}
\end{equation}
and by the convexity properties of \eqref{eq:conv}, equality holds in equation \eqref{eq:combined-bd} if{f} every set of $r-1$ vertices is contained in exactly the same number of hyperedges.  By double counting, this means that equality holds if{f} every set of $r-1$ vertices is contained in exactly $n/r$ hyperedges of $\mathcal{H}$.
\end{proof}

Combining Proposition \ref{prop:ub0or2} with Theorems \ref{thm:paley0or2} and \ref{thm:paley-edge-count} we obtain:

\begin{corollary}
For each $q \equiv 3 \pmod{4}$,  $\mathcal{H}_q$ is maximal among $4$-hypergraphs on $q+1$ points with the property that every set of 5 vertices contains at most 2 hyperedges.
\end{corollary}

Note in particular that the upper bound given in Proposition \ref{prop:ub0or2} is attained for infinitely many values of $n$.

Furthermore, we have proved Corollary \ref{cor:ex-num} which gives the following collection of exact Tur\'{a}n numbers. 

\begin{corollary}
For any prime power $q \equiv 3\pmod{4}$,
\[
\ex(q+1, \{1234, 1235, 1245\}) = \frac{(q+1)}{16} \binom{q+1}{3}.
\]
\end{corollary}

\section{Tournaments}\label{sec:tourn}

\subsection{The extended Paley tournament and $\calH_q$}

We begin with the Paley hypergraphs $\mathcal{H}_q$ considered in Section \ref{sec:design}. For any prime power $q \equiv 3 \pmod{4}$, recall that the Paley tournament, denoted here by $T(q)$, is the tournament whose vertices are elements of $\mathbb{F}_q$ where the edges are directed $x \to y$ if{f} $y-x$ is a square in $\mathbb{F}_q$.  Note that this is a well-defined since $-1$ is not a square in $\mathbb{F}_q$ when $q \equiv 3 \pmod{4}$.  We shall consider a class of tournaments that contain a Paley tournament.

\begin{definition}\label{def:t*q}
For any prime power $q \equiv 3 \pmod{4}$, define a tournament, denoted $T^*(q)$ with vertex set $V := \{(x, 1) \mid x \in \mathbb{F}_q\} \cup \{(1,0)\}$ where for every pair of vertices $(a_1, a_2), (b_1, b_2)$, the edge is directed $(a_1, a_2) \to (b_1, b_2)$ if{f} $D\left( (b_1, b_2), (a_1, a_2)\right)$ is a square in $\mathbb{F}_q$.
\end{definition}

In other words, on $\{(x, 1) \mid x \in \mathbb{F}_q\}$, the tournament $T^*(q)$ is isomorphic to the Paley tournament $T(q)$ and all edges incident to the vertex $(1, 0)$ are directed towards it. Our next goal will be to show that $\mathcal{H}_{T^*(q)}$ is isomorphic to the hypergraph $\mathcal{H}_q$ constructed in Section 1. For this, we need the following observation which is a characterisation of tournaments of the form shown in Figure \ref{fig:tourn}.

Given a tournament on four vertices $\{x_1, x_2, x_3, x_4\}$ and a cyclic permutation of the vertices $(\pi(x_1)\ \pi(x_2)\ \pi(x_3)\ \pi(x_4))$, say that a pair $\{\pi(x_j), \pi(x_{j+1})\}$ of consecutive vertices in the permutation is `reverse-oriented' with respect to the permutation if the direction of the edge in the tournament is $\pi(x_{j+1}) \to \pi(x_j)$.

\begin{fact}\label{lem:circ}
Let $T$ be a tournament on four vertices $\{a,b,c,d\}$.  Then, $T$ has the property that in any cyclic permutation of the vertices, the number of reverse-oriented pairs is odd if{f} $T$ is isomorphic to one of the two tournaments in Figure \ref{fig:tourn}.
\end{fact}

\begin{theorem}\label{thm:isotourn}
For any prime power $q \equiv 3 \pmod{4}$, the hypergraph $\mathcal{H}_{T^*(q)}$ constructed from the tournament $T^*(q)$ is isomorphic to the hypergraph $\mathcal{H}_q$.
\end{theorem}

\begin{proof}
We show that the natural mapping $\Phi$ from $\mathcal{H}_q$ to $\mathcal{H}_{T^*(q)}$ which  sends a vertex $[x:1]$ to $(x,1)$ and $[1:0]$ to $(1,0)$ induces an isomorphism.

Let $\{[a_1:a_2],[b_1:b_2],[c_1:c_2],[d_1:d_2]\}$ be an edge of $\mathcal{H}_q$ so that for any permutation $\pi$, $$S(\pi([a_1:a_2]),\pi([b_1:b_2]),\pi([c_1:c_2]),\pi([d_1:d_2]))=-1.$$  This is equivalent to saying that in the subgraph induced by the image of $\{[a_1:a_2],[b_1:b_2],[c_1:c_2],[d_1:d_2]\}$ under $\Phi$, any cyclic permutation has an odd number of reverse-oriented consecutive pairs. Now by Fact \ref{lem:circ}, this subgraph is isomorphic to one of the two graphs in Figure 1 and $\{\Phi([a_1:a_2]),\Phi([b_1:b_2]),\Phi([c_1:c_2]),\Phi([d_1:d_2])\}$ is a hyperedge in $\mathcal{H}_{T^*(q)}$.

Conversely, let $\{x, y, z, w\}$ be a non-edge in $\mathcal{H}_q$.  Then, possibly after relabelling, the edges between the four vertices are as in one of the three possibilities in Figure \ref{fig:non-edge-tourn}.

\begin{figure}[htb]
\begin{center}
\begin{tikzpicture}
	[decoration={markings, mark=at position 0.6 with {\arrow{>}}}] 
	\tikzstyle{vertex}=[circle, draw=black,  minimum size=5pt,inner sep=0pt]
	
	\node[vertex, label=left:{$x$}] at (0, 0) (x) {};
	\node[vertex, label=left:{$y$}] at (0, 1) (y) {};
	\node[vertex, label=right:{$z$}] at (1, 1) (z) {};
	\node[vertex, label=right:{$w$}] at (1, 0) (w) {};
	
	\draw[postaction={decorate}] (x) -- (y);
	\draw[postaction={decorate}] (y) -- (z);
	\draw[postaction={decorate}] (z) -- (w);
	\draw[postaction={decorate}] (w) -- (x);
\end{tikzpicture} \hspace{20pt}
\begin{tikzpicture}
	[decoration={markings, mark=at position 0.6 with {\arrow{>}}}] 
	\tikzstyle{vertex}=[circle, draw=black,  minimum size=5pt,inner sep=0pt]
	
	\node[vertex, label=left:{$x$}] at (0, 0) (x) {};
	\node[vertex, label=left:{$y$}] at (0, 1) (y) {};
	\node[vertex, label=right:{$z$}] at (1, 1) (z) {};
	\node[vertex, label=right:{$w$}] at (1, 0) (w) {};
	
	\draw[postaction={decorate}] (y) -- (x);
	\draw[postaction={decorate}] (y) -- (z);
	\draw[postaction={decorate}] (w) -- (z);
	\draw[postaction={decorate}] (w) -- (x);
\end{tikzpicture} \hspace{20pt}
\begin{tikzpicture}
	[decoration={markings, mark=at position 0.6 with {\arrow{>}}}] 
	\tikzstyle{vertex}=[circle, draw=black,  minimum size=5pt,inner sep=0pt]
	
	\node[vertex, label=left:{$x$}] at (0, 0) (x) {};
	\node[vertex, label=left:{$y$}] at (0, 1) (y) {};
	\node[vertex, label=right:{$z$}] at (1, 1) (z) {};
	\node[vertex, label=right:{$w$}] at (1, 0) (w) {};
	
	\draw[postaction={decorate}] (x) -- (y);
	\draw[postaction={decorate}] (y) -- (z);
	\draw[postaction={decorate}] (w) -- (z);
	\draw[postaction={decorate}] (x) -- (w);
\end{tikzpicture}

\end{center}
\caption{Three possible edge orientations for a non-hyperedge of $\mathcal{H}_q$}
\label{fig:non-edge-tourn}
\end{figure}
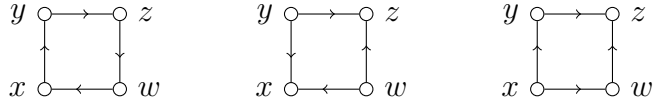

In the first case, on the left, the set $\{x, y, z, w\}$ is not a hyperedge in $\mathcal{H}_{T^*(q)}$ because no vertex has all edges directed towards or away from the other three.  In the second case, in the middle, the set is not a hyperedge in $\mathcal{H}_{T^*(q)}$ because no three vertices are part of a cyclic ordering.  Finally, in the third case on the right, the vertex $x$ could have all edges directed out, but the vertices $\{y, z, w\}$ are not a subset of a cyclic ordering.  Similarly, the vertex $z$ could have all edges directed in, but the remaining vertices are not cyclically ordered.  This shows that non-edges in $\mathcal{H}_q$ are mapped to non-edges in $\mathcal{H}_{T^*(q)}$.

Thus, the two hypergraphs are isomorphic.
\end{proof}

Note that, in light of the remarks at the start of this section, Theorem \ref{thm:isotourn} supplies a different proof of the fact that in $\mathcal{H}_q$ any set of 5 vertices span 0 or 2 edges.

\subsection{Switching tournaments}\label{sec:switching}

Recall from the introduction that Baber's construction associates to each tournament $T$ a 4-hypergraph $\calH_T$ with the property that 5 vertices span 0 or 2 hyperedges. In this section we consider the opposite problem of associating a tournament to any hypergraph satisfying this condition. 

The notion, examined by Frankl and F\"{u}redi \cite{FF84}, of $3$-uniform hypergraphs in which every $4$ vertices span $0$ or $2$ hyperedges is a special case of what is called a `two-graph' (not to be confused with a graph).  Two-graphs were introduced by Higman (see \cite{dT77}) and are defined to be $3$-uniform hypergraphs with the property that every set of $4$ vertices spans an even number of hyperedges.  As described by Cameron and van Lint \cite{CL}, a two-graph can be constructed from a graph $G = (V, E)$ by defining a hypergraph on $V$ whose hyperedges are the sets of $3$ vertices that contain an odd number of edges in $G$.  Furthermore, every two-graph arises from such a construction.  A survey on two-graphs was given by Seidel and Taylor \cite{ST81}.

For example, the $3$-uniform hypergraph on $6$ vertices with $10$ hyperedges given by Frankl and F\"{u}redi in \cite{FF84} corresponds to the graph shown in Figure \ref{fig:two-graph}.  Note that the graph in Figure \ref{fig:two-graph} is the Paley graph for $\mathbb{F}_5$ with an additional isolated vertex (labeled $5$).  This is the only such construction from a Paley graph with the property that every set of $4$ vertices contain either $0$ or $2$ subsets of size $3$ that span an odd number of edges.  Indeed, all other Paley graphs either contain a copy of $K_4$ or else an induced subgraph consisting of $K_3$ and an isolated vertex.

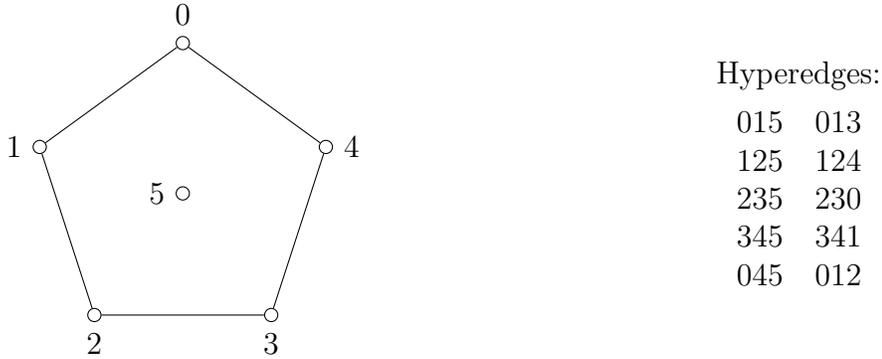
\begin{figure}[htb]
\begin{minipage}{0.5\linewidth}
\begin{center}
\begin{tikzpicture}
	\tikzstyle{vertex}=[circle, draw=black,  minimum size=5pt,inner sep=0pt]
	\node[vertex, label=left:{$5$}] at (0, 0) (5) {};
	\node[vertex, label=right:{$4$}] at (18:2) (4) {};
	\node[vertex, label=above:{$0$}] at (90:2) (0) {};
	\node[vertex, label=left:{$1$}] at (162:2) (1) {};
	\node[vertex, label=below:{$2$}] at (234:2) (2) {};
	\node[vertex, label=below:{$3$}] at (306:2) (3) {};
	\draw (0) -- (1) -- (2) -- (3) -- (4) -- (0); 
\end{tikzpicture}
\end{center}
\end{minipage}
\begin{minipage}{0.5\linewidth}
\begin{center}
Hyperedges:\\ \vspace*{5pt}
\begin{tabular}{ll}
015	&013\\
125	&124\\
235	&230\\
345 &341\\
045 &012
\end{tabular}
\end{center}
\end{minipage}
\caption{Two-graph representation of $3$-uniform hypergraph from \cite{FF84}}
\label{fig:two-graph}
\end{figure}

In \cite{CAM77} Cameron introduced the notion of an \textit{oriented two-graph} which, like the hypergraphs $\calH_T$ may also be associated to a tournament $T$. Indeed, if $T$ is regarded as an antisymmetric function $f$ from ordered pairs of distinct vertices to $\{ \pm 1\}$ (where $f(x, y) = 1$ if and only if there is a directed edge from $x$ to $y$) then the associated oriented two-graph is given by the function $g$ defined on ordered triples of distinct vertices $x,y,z$ as follows: $$g(x,y,z)=f(x,y)f(y,z)f(z,x)$$ (see \cite[Section 2]{BC2000}). We contrast this to the definition of $\calH_T$ which may be regarded as a function defined on unordered quadruples of distinct vertices. 

The following operation on tournaments was introduced by Moorhouse \cite{EM95} in connection to oriented two-graphs.

\begin{definition}\label{def:switch}
Given a tournament $T$ on vertex set $V$ and a set $A \subseteq V$, then \emph{$T$ switched with respect to $A$} is the tournament obtained from $T$ by reversing the orientation of all edges between $A$ and $V \setminus A$.

Two tournaments $T_1$ and $T_2$ both on vertex set $V$ are said to be \emph{switching equivalent} if{f} there exists $A \subseteq V$ so that $T_2$ is precisely $T_1$ switched with respect to $A$.
\end{definition}

Moorhouse \cite{EM95} also notes that if $M$ is the $(0, \pm 1)$-adjacency matrix of a tournament $T$ and $T'$ is a switching equivalent tournament, then there is a $(\pm 1)$-diagonal matrix $D$ so that the $(0, \pm 1)$-adjacency matrix of $T'$ is $DMD$.  Note that switching with respect to the empty set (or equivalently the entire vertex set) is a legal operation, but leaves the tournament unchanged.  One can further verify directly that being switching equivalent is indeed an equivalence relation.  To show transitivity, note that switching with respect to a set $A$ and then switching with respect to a set $B$ corresponds to switching with respect to the set $(A \cap B) \cup (A^c \cap B^c)$.

The notion of switching for tournaments is closely related to a concept of switching in graph theory (see \cite[Chapter 4]{CL}). 

It was shown by Moorhouse \cite{EM95} that two tournaments are switching equivalent if and only if they induce the same oriented two-graph. In one direction, an analogous statement holds for the hypergraphs $\calH_T$.

\begin{lemma}\label{l:tt'switch}
Let $T,T'$ be two tournaments which are switching equivalent. Then $\calH_T=\calH_{T'}$.
\end{lemma}

\begin{proof}
It suffices to consider the effect of switching on tournaments on $4$ vertices.  Note that for any cyclic permutation of $4$ vertices $(x_1\ x_2\ x_3\ x_4)$, any switching operation changes the orientation of the edges between either $0$, $2$ or $4$ of the pairs $\{x_1, x_2\}, \{x_2, x_3\}, \{x_3, x_4\}$, and $\{x_4, x_1\}$.  Thus, if the number of reverse-oriented edges is odd before switching, it remains odd after switching also.  Thus, by Fact \ref{lem:circ}, both edges and non-edges in hypergraphs $\calH_T$ are preserved by switching.
\end{proof}

The converse to Lemma \ref{l:tt'switch} fails because the first tournament in Figure 1 is clearly not switching equivalent to the tournament obtained by reversing all edges. 

In fact, even if one allows the additional operation of ``reversing all the edges,'' the converse still fails as the two tournaments on 5 vertices, given in Figure \ref{fig:not-switching} show.  The two tournaments differ only in the orientation of the edge between vertices $1$ and $5$ and yet both yield the $4$-uniform hypergraph $\{1234, 2345\}$.

\begin{figure}[htb]
\begin{center}
\begin{tikzpicture}
	[decoration={markings, mark=at position 0.6 with {\arrow{>}}}] 
	\tikzstyle{vertex}=[circle, draw=black,  minimum size=5pt,inner sep=0pt]
	
	\node[vertex, label=above:{$1$}] at (90:2) (1) {};
	\node[vertex, label=right:{$2$}] at (18:2) (2) {};
	\node[vertex, label=below:{$3$}] at (306:2) (3) {};
	\node[vertex, label=below:{$4$}] at (234:2) (4) {};
	\node[vertex, label=left:{$5$}] at (162:2) (5) {};
	
	\draw[postaction={decorate}] (1) -- (2);
	\draw[postaction={decorate}]  (1) -- (3);
	\draw[postaction={decorate}]  (1) -- (4);
	\draw[postaction={decorate}, ultra thick]  (1) -- (5);
	\draw[postaction={decorate}]  (3) -- (2);
	\draw[postaction={decorate}] (2) -- (4);
	\draw[postaction={decorate}]  (5) -- (2);
	\draw[postaction={decorate}]  (4) -- (3);
	\draw[postaction={decorate}] (5) -- (3);
	\draw[postaction={decorate}] (5) -- (4);	
\end{tikzpicture} \hspace{20pt}
\begin{tikzpicture}
	[decoration={markings, mark=at position 0.6 with {\arrow{>}}}] 
	\tikzstyle{vertex}=[circle, draw=black,  minimum size=5pt,inner sep=0pt]
	
	\node[vertex, label=above:{$1$}] at (90:2) (1) {};
	\node[vertex, label=right:{$2$}] at (18:2) (2) {};
	\node[vertex, label=below:{$3$}] at (306:2) (3) {};
	\node[vertex, label=below:{$4$}] at (234:2) (4) {};
	\node[vertex, label=left:{$5$}] at (162:2) (5) {};
	
	\draw[postaction={decorate}] (1) -- (2);
	\draw[postaction={decorate}]  (1) -- (3);
	\draw[postaction={decorate}]  (1) -- (4);
	\draw[postaction={decorate}, ultra thick]  (5) -- (1);
	\draw[postaction={decorate}]  (3) -- (2);
	\draw[postaction={decorate}] (2) -- (4);
	\draw[postaction={decorate}]  (5) -- (2);
	\draw[postaction={decorate}]  (4) -- (3);
	\draw[postaction={decorate}] (5) -- (3);
	\draw[postaction={decorate}] (5) -- (4);	
\end{tikzpicture}
\end{center}
\caption{Two tournaments that determine the hypergraph $\{1234, 2345\}$}
\label{fig:not-switching}
\end{figure}
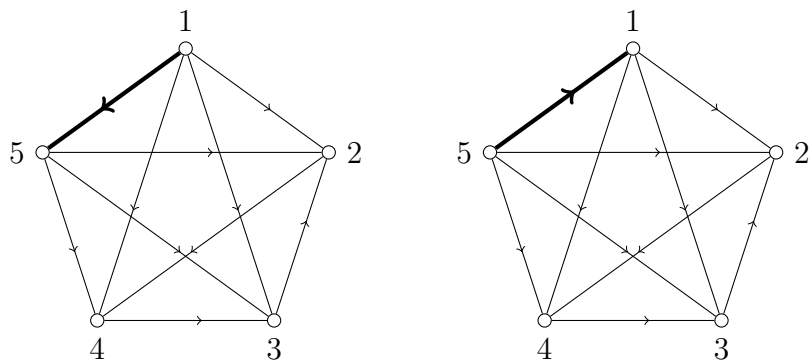

The following application of switching is used in further results to come and shows that the construction of the extended Paley tournament with one vertex having all incident edges directed towards it is not a particularly unusual condition.

\begin{lemma}\label{lem:univ}
Let $T$ be a tournament on vertex set $V$ and let $w \in V$.  There is a tournament $T'$ that is switching equivalent to $T$ in which all edges incident to $w$ are directed towards $w$.
\end{lemma}

\begin{proof}
Let $A$ be the set of vertices in $a \in V\setminus \{w\}$ for which the edge between $a$ and $w$ is directed away from $w$.  Switching $T$ with respect to $A$ yields a tournament with all edges incident to $w$ directed towards it.
\end{proof}

Similarly, one could obtain a tournament with all edges directed away from a particular vertex.

As an application of switching, we have the following result which gives a \textit{necessary} condition on a 4-hypergraph $\calH$ for the existence of a tournament $T$ with $\calH=\calH_T$:

\begin{proposition}\label{p:oddcycle}
Let $\calH$ be a 4-hypergraph with the property that $\calH=\calH_T$ for some tournament $T$ with vertex set $V$. Then for each $u,v \in V$ the graph on $T \backslash \{u,v\}$ with edge set
\[
L_\calH(u,v):=\{\{x,y\} \mid x,y \in V \backslash \{u,v\}, \{u,v,x,y\} \in \calH\}
\]
(the link of $\{u,v\}$) is a bipartite graph on $V \backslash \{u,v\}$.
\end{proposition}

\begin{proof}
Suppose that $\calH=\calH_T$ for some tournament $T$ with vertex set $V$.  By Lemma \ref{lem:univ}, we can assume that all edges incident to $u$ are directed towards $u$ in $T$.  If $N_{\calH}(u, v)$ contains no cycles, then it is a tree and so bipartite.  Suppose that $L_{\calH}(u, v)$ contains a cycle with vertices $a_1, a_2, a_3, \ldots, a_k$ (in that order).  

Consider now the orientations of edges in $T$ between consecutive vertices in the cycle and the vertices of the cycle and $v$.  Let $i \leq k-1$ and suppose that the edge between $a_i$ and $a_{i+1}$ is directed $a_i \to a_{i+1}$.  Since  $\{u, v, a_i, a_{i+1}\} \in \calH_T$ and all edges incident to $u$ are directed towards $u$, then the remaining edges incident to $v$ are directed $v \to a_i$ and $a_{i+1} \to v$.  Similarly, if the edge between $a_i$ and $a_{i+1}$ is directed $a_{i+1} \to a_i$, then the edges incident to $v$ are directed $a_i \to v$ and $v \to a_{i+1}$.  In either case, the vertices of the cycle alternate between being in-neighbours of $v$ and out-neighbours of $v$.  Since this also holds for the pair $\{a_{k}, a_1\}$, then $k$ is even.

Thus, if the graph $L_{\calH}(u,v)$ contains any cycles, they are even cycles and hence $L_{\calH}(u,v)$ is a bipartite graph.
\end{proof}

The converse of Proposition \ref{p:oddcycle} is not true.  For example, the $4$-uniform hypergraph given by sets in \eqref{eq:non-tourn} has the property that every $5$-set of vertices spans either $0$ or $2$ hyperedges and the neighbourhood graph of every pair of vertices is bipartite, but one can verify directly that the hypergraph can not be represented by a tournament.
\begin{equation}\label{eq:non-tourn}
\begin{tabular}{llll}
	$\{6, 7, 8, 11 \}$,  &$\{6, 7, 9, 12\}$, &$\{6, 7, 10, 11\}$,  &$\{6, 7, 10, 12\}$,\\
	$\{6, 8, 9, 12\}$,  &$\{6, 9, 10, 11\}$, &$\{6, 9, 11, 12\}$, &$\{7, 8, 9, 11\}$,\\
	$\{7, 8, 10, 12\}$,  &$\{7, 8, 11, 12\}$, &$\{8, 9, 10, 11\}$, &$\{8, 9, 10, 12\}$
\end{tabular}
\end{equation}
The hypergraph given in Equation \eqref{eq:non-tourn} consists of the hyperedges in $\mathcal{M}$ that contain none of the vertices from $\{1, 2, 3, 4, 5\}$.

Proposition \ref{p:oddcycle} can be used to show that even extremal hypergraphs for the property that any $5$ vertices span either $0$ or $2$ hyperedges need not arise from tournaments.

\begin{corollary}\label{cor:M11-notourn}
There is no tournament $T$ such that $\mathcal{M} = \mathcal{H}_T$.
\end{corollary}

\begin{proof}
For any two vertices $a, b \in [1, 12]$, consider the graph on $V(\calM) \backslash \{a,b\}$ given by the link of $\{a,b\}$. Thus the edges consist of pairs $$L_{\calM}(a, b)=\left\{\{c, d\} \mid c,d \in V(\calM) \backslash \{a,b\},  \{a,b,c, d\} \in \calM\right\}.$$  By \cite[Lemma 4.8]{DGHP} (or alternatively via a direct GAP \cite{GAP} computation), $L_{\calM}(a, b)$ is isomorphic to the Petersen graph, which is not bipartite. Thus, by Proposition \ref{p:oddcycle}, there is no tournament $T$ with $\calM = \mathcal{H}_T$.
\end{proof}

\section{Open questions}\label{sec:open}
Note that using the recursive construction of Frankl and F\"{u}redi \cite{FF84} with a Paley hypergraph, one can show that for any $\varepsilon >0$ and sufficiently large $n$, there is a $4$-uniform hypergraph with $\frac{1}{4}\binom{n}{4}(1-\varepsilon)$ hyperedges with the property that any $5$ vertices span at most $2$ hyperedges. This fact can also be deduced since for any $\varepsilon > 0$ and $n$ sufficiently large, there is a prime $q \equiv 3 \pmod{4}$ with $q \in [n, (1+\varepsilon)n]$.  
However, for divisibility reasons, if the upper bound given in Proposition \ref{prop:ub0or2} is attained, then the number of vertices in the graph is divisible by $4$.

\begin{question}\label{qn:existence}
For which natural numbers $n \equiv 0 \pmod 4$ does there exist a $3-(n,4,n/4)$ design with the property that every set of 5 vertices spans either 0 or 2 hyperedges? 
\end{question} 

One hint that $n$ must be of the form $q+1$ ($q$ a prime power) comes from the observation that all known extremal examples  have automorphism groups which are 3-transitive and by \cite[Tables 7.3 and 7.4]{CAM} the only such permutation groups $(G,n)$ (with $G$ acting on $n$ points) are $(\M_{11},12)$, $(\M_{22},22)$ and extensions of $(\PSL(2,q),q+1)$. We note that an easy GAP \cite{GAP} computation reveals that $\M_{22}$ does not have an orbit of the form we require.

In general, Question~\ref{qn:existence} may be quite difficult to answer.  The construction given in this paper does not seem to generalize naturally to values of $n$ which are not one more than a prime power.  

One might further ask for an improved upper bound in the cases when $n$ is not divisible by $4$.  A slight improvement in Proposition \ref{prop:ub0or2} follows immediately by convexity. \newline
\newline
While a complete classification of 4-hypergraphs with the property that every set of 5 vertices spans either 0 or 2 hyperedges (which parallels that given for 3-hypergraphs in \cite{FF84}) appears difficult, it may be of interest to compare the ways in which these hypergraphs arise. As previously mentioned, two natural sources of examples are given by:

\begin{itemize}
\item[(i)] finite subsets of $S^2$ (where edges are given by 4-subsets of points whose convex hull contains the origin);
\item[(ii)] the hypergraphs $\calH_T$ where $T$ is a tournament.
\end{itemize}

\begin{question}
What is the relationship between these two families of 4-hypergraphs? For example, can every hypergraph which arises from points on the unit sphere be realized using a tournament?
\end{question}

Finally, we make some remarks about future work \cite{GS2}. Since writing this paper, the authors made two key observations. 

Firstly, there are other switching classes $\mathcal{S}$ of tournaments on $r$-vertices (apart from that given in Figure \ref{fig:tourn}) with the property that 0 or 2 subtournaments of any tournament on $r+1$ vertices lie in $\mathcal{S}$. For example when $r=6$, it can be shown that the switching class of the tournament in Figure \ref{fig:tourn2} has this property. Thus by taking a random tournament on $n$ vertices and considering only those subsets of the vertex set which induce tournaments in $\mathcal{S}$ we can obtain new bounds for Tur\'{a}n densities. For example using the switching class $\mathcal{S}$ of the tournament in Figure \ref{fig:tourn2} it is shown in \cite{GS2} that

\begin{equation}\label{eq:density}
	\frac{9}{64} \leq  \pi(\{123456, 123457, 123467\}) \leq \frac{1}{6}.
\end{equation}

Secondly, let $T^*(q)$ be the tournament of Definition \ref{def:t*q} the vertices of which are acted naturally upon by $\PGL(2,q)$. Then for each $3 \le r \le q+1$ and switching class $\mathcal{S}$ of tournaments on $r$ vertices, it can be shown that the set of $r$-subsets of $V(T^*(q))$ which induce subtournaments of $T^*(q)$ in $\mathcal{S}$ is a union of orbits of  $\PGL(2,q)$. Theorem \ref{thm:main-const} (c) follows as a special case of this observation (on taking $r=4$ and $\mathcal{S}$ equal to the switching class of Figure \ref{fig:tourn}) and the hypergraph so obtained provided us with infinitely many Tur\'{a}n numbers. Surprisingly, different choices of $q$ and $\mathcal{S}$, yield further Tur\'{a}n numbers. For example in \cite{GS2} it is shown that when $q=11$ there are exactly $264$ 6-subsets of $V(T^*(11))$ 
which induce subtournaments in the switching class of the tournament in Figure \ref{fig:tourn2}. This observation, combined with the remarks in the previous paragraph and Proposition \ref{prop:ub0or2}, yields:

\[
\ex(12, \{123456, 123457, 123467\}) = 264.
\]

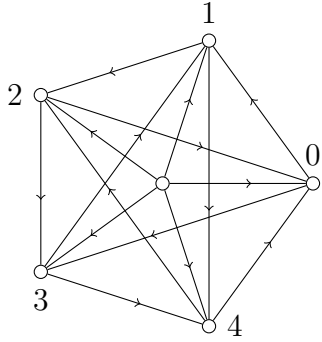
\begin{figure}[htb]
\begin{center}
\begin{tikzpicture}
	[decoration={markings, mark=at position 0.6 with {\arrow{>}}}] 
	\tikzstyle{vertex}=[circle, draw=black,  minimum size=5pt,inner sep=0pt]
	
	\node[vertex, label=above:{$0$}] at (2, 0) (y) {};
	\node[vertex, label=above:{$1$}] at (2*0.309, 2*0.95) (z) {};
	\node[vertex, label=left:{$2$}] at (2*-0.809, 2*0.588) (w) {};
    \node[vertex, label=below:{$3$}] at (2*-0.809, 2*-0.588) (a) {};
	\node[vertex, label=right:{$4$}] at (2*0.309, 2*-0.95) (b) {};	
    \node[vertex, label=right:{}] at (0, 0) (c) {};		
	
	\draw[postaction={decorate}] (y) -- (z);
	\draw[postaction={decorate}] (z) -- (w);
	\draw[postaction={decorate}] (w) -- (a);
	\draw[postaction={decorate}] (a) -- (b);
	\draw[postaction={decorate}] (b) -- (y);
	
	\draw[postaction={decorate}] (z) -- (b);
	\draw[postaction={decorate}] (b) -- (w);
	\draw[postaction={decorate}] (w) -- (y);
	\draw[postaction={decorate}] (y) -- (a);
	\draw[postaction={decorate}] (a) -- (z);
	
	\draw[postaction={decorate}] (c) -- (b);
	\draw[postaction={decorate}] (c) -- (w);
	\draw[postaction={decorate}] (c) -- (y);
	\draw[postaction={decorate}] (c) -- (a);
	\draw[postaction={decorate}] (c) -- (z);
	
\end{tikzpicture} \hspace{10pt}
\end{center}
\caption{A special tournament on 6 vertices}
\label{fig:tourn2}
\end{figure}

\section*{Acknowledgement}

The authors wish to thank John Talbot for the interesting discussion on this topic and Rahil Baber for sharing his tournament construction.  We would also like to thank Jonathan Bober for showing us the usefulness of the identity in equation \eqref{eq:sumLegpairs} for counting certain sets of squares in prime fields.

\section*{Appendix: Hypergraph $\mathcal{M}$ on $12$ vertices}\label{sec:M112}

\begin{tabular}{lllll}
$\{1, 2, 3, 7\}$     &$\{1, 5, 8, 9\}$     &$\{1, 4, 5, 6\}$     &$\{2, 5, 7, 11\}$     &$\{4, 5, 7, 10\}$\\
$\{1, 3, 8, 10\}$     &$\{3, 8, 9, 11\}$     &$\{1, 2, 8, 9\}$     &$\{2, 5, 6, 7\}$     &$\{3, 4, 6, 11\}$\\
$\{1, 7, 9, 11\}$     &$\{2, 4, 5, 12\}$     &$\{3, 4, 7, 9\}$     &$\{2, 3, 6, 10\}$     &$\{2, 4, 7, 10\}$\\
$\{3, 5, 8, 10\}$     &$\{5, 8, 10, 11\}$     &$\{1, 6, 8, 10\}$     &$\{4, 8, 10, 11\}$     &$\{6, 8, 9, 12\}$\\
$\{4, 5, 9, 10\}$     &$\{7, 8, 9, 11\}$     &$\{1, 6, 7, 9\}$     &$\{2, 3, 4, 6\}$     &$\{2, 3, 4, 7\}$\\
$\{2, 3, 4, 5\}$     &$\{2, 6, 8, 11\}$     &$\{3, 6, 7, 12\}$     &$\{1, 4, 5, 11\}$     &$\{1, 4, 10, 11\}$\\
$\{1, 2, 9, 11\}$     &$\{1, 4, 9, 12\}$     &$\{2, 7, 9, 12\}$     &$\{1, 2, 4, 8\}$     &$\{2, 5, 6, 9\}$\\
$\{4, 6, 7, 9\}$     &$\{3, 6, 8, 9\}$     &$\{3, 5, 6, 10\}$     &$\{1, 2, 3, 9\}$     &$\{2, 6, 7, 10\}$\\
$\{3, 7, 8, 10\}$     &$\{3, 5, 8, 9\}$     &$\{2, 6, 8, 10\}$     &$\{3, 7, 10, 11\}$     &$\{2, 8, 10, 12\}$\\
$\{2, 3, 10, 11\}$     &$\{7, 8, 10, 12\}$     &$\{4, 6, 10, 12\}$     &$\{4, 8, 9, 12\}$     &$\{3, 7, 9, 10\}$\\
$\{3, 6, 9, 10\}$     &$\{2, 7, 9, 10\}$     &$\{1, 2, 4, 10\}$     &$\{2, 4, 9, 11\}$     &$\{1, 3, 4, 9\}$\\
$\{2, 6, 8, 9\}$     &$\{4, 5, 6, 7\}$     &$\{3, 5, 6, 7\}$     &$\{2, 7, 8, 9\}$     &$\{6, 9, 10, 11\}$\\
$\{1, 4, 7, 8\}$     &$\{3, 5, 8, 12\}$     &$\{1, 6, 8, 12\}$     &$\{1, 3, 7, 11\}$     &$\{1, 3, 6, 11\}$\\
$\{1, 2, 7, 12\}$     &$\{2, 3, 5, 11\}$     &$\{1, 6, 10, 11\}$     &$\{1, 7, 10, 12\}$     &$\{4, 5, 9, 11\}$\\
$\{1, 5, 9, 12\}$     &$\{6, 9, 11, 12\}$     &$\{2, 5, 9, 10\}$     &$\{1, 6, 7, 8\}$     &$\{6, 7, 8, 11\}$\\
$\{1, 5, 7, 8\}$     &$\{2, 3, 5, 9\}$     &$\{1, 2, 4, 6\}$     &$\{2, 4, 6, 9\}$     &$\{4, 6, 8, 10\}$\\
$\{2, 5, 8, 10\}$     &$\{4, 5, 8, 9\}$     &$\{3, 6, 7, 8\}$     &$\{3, 4, 5, 10\}$     &$\{1, 5, 6, 10\}$\\
$\{5, 7, 9, 11\}$     &$\{1, 5, 6, 9\}$     &$\{1, 3, 6, 9\}$     &$\{4, 7, 8, 10\}$     &$\{1, 3, 8, 11\}$\\
$\{1, 2, 5, 10\}$     &$\{3, 9, 10, 12\}$     &$\{5, 6, 9, 11\}$     &$\{1, 5, 7, 10\}$     &$\{5, 6, 10, 12\}$\\
$\{1, 3, 10, 12\}$     &$\{2, 3, 6, 12\}$     &$\{2, 4, 10, 12\}$     &$\{1, 3, 4, 8\}$     &$\{4, 6, 9, 10\}$\\
$\{3, 4, 6, 8\}$     &$\{6, 7, 10, 11\}$     &$\{5, 7, 9, 12\}$     &$\{2, 4, 5, 8\}$     &$\{4, 7, 8, 9\}$\\
$\{1, 2, 6, 7\}$     &$\{1, 4, 9, 10\}$     &$\{2, 9, 10, 11\}$     &$\{1, 2, 8, 12\}$     &$\{4, 8, 11, 12\}$\\
$\{1, 5, 8, 11\}$     &$\{1, 2, 5, 11\}$     &$\{1, 4, 6, 12\}$     &$\{5, 6, 8, 11\}$     &$\{1, 3, 6, 12\}$\\
$\{2, 7, 11, 12\}$     &$\{3, 5, 6, 11\}$     &$\{1, 3, 5, 7\}$     &$\{2, 5, 6, 12\}$     &$\{1, 3, 5, 12\}$\\
$\{2, 10, 11, 12\}$     &$\{2, 4, 7, 11\}$     &$\{4, 5, 7, 12\}$     &$\{3, 9, 11, 12\}$     &$\{4, 6, 11, 12\}$\\
$\{2, 4, 8, 11\}$     &$\{3, 4, 8, 12\}$     &$\{2, 3, 8, 11\}$     &$\{3, 5, 7, 9\}$     &$\{4, 6, 7, 11\}$\\
$\{2, 3, 7, 8\}$     &$\{3, 4, 10, 11\}$     &$\{1, 2, 3, 10\}$     &$\{1, 4, 7, 11\}$     &$\{2, 4, 9, 12\}$\\
$\{3, 4, 9, 11\}$     &$\{4, 5, 6, 8\}$     &$\{5, 9, 10, 12\}$     &$\{2, 3, 9, 12\}$     &$\{3, 4, 10, 12\}$\\
$\{5, 10, 11, 12\}$     &$\{6, 7, 9, 12\}$     &$\{6, 7, 10, 12\}$     &$\{5, 7, 10, 11\}$     &$\{2, 5, 7, 8\}$\\
$\{1, 4, 7, 12\}$     &$\{1, 3, 4, 5\}$     &$\{1, 7, 9, 10\}$     &$\{7, 8, 11, 12\}$     &$\{1, 2, 5, 12\}$\\
$\{2, 3, 8, 12\}$     &$\{5, 6, 8, 12\}$     &$\{4, 5, 11, 12\}$     &$\{1, 9, 11, 12\}$     &$\{3, 4, 7, 12\}$\\
$\{5, 7, 8, 12\}$     &$\{3, 5, 11, 12\}$     &$\{1, 2, 6, 11\}$     &$\{2, 6, 11, 12\}$     &$\{3, 7, 11, 12\}$\\
$\{8, 9, 10, 11\}$     &$\{1, 8, 9, 10\}$     &$\{1, 10, 11, 12\}$     &$\{8, 9, 10, 12\}$     &$\{1, 8, 11, 12\}$
\end{tabular}

\end{document}